\documentclass[12pt]{amsart}
\usepackage[utf8]{inputenc} 
\usepackage{amssymb}
\usepackage{amsfonts}
\usepackage{amsmath}
\usepackage{array}
\usepackage{enumitem}
\usepackage{tikz}
\usepackage{breqn}
\usepackage{multicol}
\usepackage{enumitem}
\usepackage{amsthm}
\RequirePackage[top=1in,bottom=1in,left=1.4in,right=1.4in,includehead]{geometry}

\usepackage[hyperindex=true,frenchlinks=true,colorlinks=true,
  linktocpage,
]{hyperref}

\newtheorem{theorem}{Theorem}[section]
\newtheorem{lemma}[theorem]{Lemma}
\newtheorem{proposition}[theorem]{Proposition}
\newtheorem{corollary}[theorem]{Corollary}

\def\SG{{\mathfrak S}}               
\def\SP{{\mathfrak P}}               
\def\des{\operatorname{des}}         
\def\seg{\operatorname{seg}}         
\def\SCDes{\operatorname{SDes}}      
\def\Des{\operatorname{Des}}         
\def\Bar{\operatorname{Seg}}         
\def\segn{|\!|\!\!=}                 
\def\std{\operatorname{std}}         

\def\RSC{{\it R}}              
\def\SSC{{\it S}}              
\def\GSP{{\bf G}}              
\def\Acal{\mathcal{A}}         

\def\SCQSym{{\bf SCQSym}}      
\def\SPQSym{{\bf SPQSym}}      
\def\NCSF{{\bf Sym}}           
\def\FQSym{{\bf FQSym}}        

\newcommand{\OEIS}[1]{\href{http://oeis.org/#1}{{#1}}}
\def\ie{{\emph i.e.}}          

\title{Eulerian polynomials on segmented permutations}
\author{Arthur Nunge}


\usepackage[maxbibnames=99, backend=bibtex]{biblatex}
\begin{document}
\setlength{\abovedisplayskip}{6pt}
\setlength{\belowdisplayskip}{6pt}
\setlength{\abovedisplayshortskip}{6pt}
\setlength{\belowdisplayshortskip}{6pt} 

\maketitle
\begin{abstract}
  We define a generalization of the Eulerian
  polynomials and the Eulerian numbers by considering a descent
  statistic on segmented permutations coming from the study of
  2-species exclusion processes and a change of basis in a Hopf
  algebra. We give some properties satisfied by these generalized
  Eulerian numbers. We also define a 
  $q$-analog of these Eulerian polynomials which gives back usual
  Eulerian polynomials and ordered Bell polynomials for specific values
  of its variables. We also define a noncommutative analog 
  living in the algebra of segmented compositions. It gives us an
  explicit generating function and some identities satisfied by the
  generalized Eulerian polynomials such as a Worpitzky-type relation. 
\end{abstract}

\section*{Introduction}

The \emph{Eulerian numbers} $A(n,k)$ are involved in many combinatorial
problems. They count the
number of permutations $\sigma \in \SG_n$ with $k$ descents, \ie,
positions of the values followed by a smaller value in $\sigma$
written as a word. These numbers define a natural $t$-analog of $n!$,
the \emph{Eulerian polynomials} whose coefficient are the $A(n,k)$:
\begin{equation*}
  A_n(t) = \sum_{\sigma \in \SG_n}t^{\des(\sigma)}.
\end{equation*}
These polynomials and numbers have been extensively studied over the
years, see~\cite{Pet} for an overview.

Another approach is to consider these polynomials as a formal sum of
permutations living in the algebra of free quasi-symmetric functions
($\FQSym$) whose bases are indexed by permutations~\cite{FQSym}. In this
case, the polynomials are in fact living in a subalgebra of $\FQSym$:
the noncommutative symmetric functions ($\NCSF$) defined
in~\cite{NCSF}. Studying this noncommutative analogs easily gives back
many properties of the Eulerian numbers, see~\cite{NT2}.

More recently, in order to refine the probabilities arising from the
study of the $2$-species exclusion processes, the authors of~\cite{CN}
defined a recoil statistic on partially signed permutations. With a
natural notion of inverse on these objects, this statistic corresponds
to a descent statistic on \emph{segmented permutations} $\SP_n$,
permutations where values can be separated by vertical bars. Moreover,
the authors defined a generalization of $\FQSym$ and its relation with
a known generalization of $\NCSF$: the segmented composition
quasi-symmetric functions algebra ($\SCQSym$), defined
in~\cite{SCQSym}, using this descent statistic.

In this paper we study this descent statistic on segmented permutations
through a generalization of the
Eulerian numbers $T(n,k)$ alongside with a natural refinement
$K(n,i,j)$ counting the number of segmented permutations having $i$
descents and $j$ bars. They define natural $t$-analogs and
$(q,t)$-analogs of $2^{n-1}n!$, the number of segmented permutations of
size $n$:
\begin{equation}
  P_n(t) = \sum_{\sigma \in \SP_n}t^{\des(\sigma)},
\end{equation}
\begin{equation}
  \alpha_n(t,q) = \sum_{\sigma\in \SP_n}t^{\des(\sigma)}q^{\seg(\sigma)}.
\end{equation}
In addition to generalizing the Eulerian polynomials in their definition
and some of their properties, these polynomials involve the ordered
Bell numbers and the Stirling numbers of the second kind that play an
important role in combinatorics. Moreover, the
$\alpha_n$ gives back the usual Eulerian polynomials at $q=0$ and
the ordered Bell polynomials at $t=0$.
Then we define a noncommutative analog $\Acal_n(t,q)$ of the
polynomials $\alpha_n$ using the ribbon basis of $\SCQSym$.
\begin{equation}
  \Acal_n(t,q) = \sum_{I\segn n}t^{\des(I)}q^{\seg(I)}R_I.
\end{equation}
Considering the expression of the $\Acal_n$ on the complete basis we
define below, we obtain in Theorem~\ref{th:GF} an explicit
expression of the generating function of the~$\alpha_n$.
As a consequence, we obtain in
Proposition~\ref{prop:Dobin} an identity satisfied by our
generalized Eulerian polynomials which is very similar to the one
satisfied by the usual Eulerian polynomials.
This identity can be seen as a $t$-analog of the
Dobi\'nski-type relation satisfied by the ordered Bell numbers in the
sense of~\cite{BPS}.

Finally, we use the generating function of the polynomials $\alpha_n$
to prove Theorem~\ref{th:Worp}, a generalization of Worpitzky's
identities expressing the discrete derivatives of the $n$-th power of
an integer in terms of the coefficients of $\alpha_n$.

We give all detailed definitions in Section~\ref{sec:def}. In
Section~\ref{sec:numbers}, we define the generalized Eulerian numbers
and polynomials with some elementary propositions. In
section~\ref{sec:alge} we use a noncommutative analog of the
generalized Eulerian polynomials to prove the main results of the 
paper.

\section{Definitions and background}
\label{sec:def}
\subsection{Segmented compositions and permutations}

A \emph{segmented composition} of an integer $n$ (denoted by
$I\segn n$) is a finite sequence $I=(i_1, \cdots, i_r)$ of positive
integers separated by commas or bars that sum to $n$. In the following
examples, when there is no ambiguity we do not represent the
commas. The integer $r$  is called the \emph{length} of 
the segmented composition, denoted by $\ell(I)$. We denote by $\seg(I)$
the number of bars in $I$ and by $\des(I) = \ell(I) - \seg(I)$ the number
of values that are not followed by a bar. The \emph{descent
  set} of a segmented composition $I$, denoted by $\Des(I)$, is the
set of values $i_1+i_2+\cdots i_{k}$ for $k<r$ where $i_k$ is not
followed by a bar in $I$. Similarly, the \emph{segmentation
  set} of $I$, denoted by $\Bar(I)$, is the set of values
$i_1+i_2+\cdots i_{k}$ for $k<r$ where $i_k$ is followed by a
bar in $I$. Note that $|\Des(I)|=\des(I)-1$ and
$|\Bar(I)|=\seg(I)$. For example, with $I = 21|2|31$,
$(\Des(I), \Bar(I))= (\{2,8\},\{3,5\})$. We shall use that when $n$ is
fixed, any segmented composition $I$ is in bijection with the pair 
of sets $(\Des(I), \Bar(I))$.
We define the \emph{reverse refinement order} on segmented
compositions by $I \succeq K$ if and only if $I$ and $K$ 
are two segmented composition of the same integer and
\begin{equation}
  \begin{split}
    \Des(I) \supseteq & \Des(K),\\
    \Bar(I) \subseteq \Bar(K) \subseteq & \Big{(}\Des(I) \cup \Bar(I)\Big{)}.\\
  \end{split}
\end{equation}
In this case we say that $I$ is finer than $K$.
For example, $213|22|121 \succeq 24|22|1|3$.

Let $I$ and $K$ be two segmented compositions. Then $I\cdot K$
represent their concatenation with a comma between the last value of $I$ and
the first value of $K$, $I\triangleright K$ represent their
concatenation where the last value of $I$ and the first value of $K$
are added together, and $I|K$ represent their concatenation where the
last value of $I$ is followed by a bar. For example, with $I=21|1$ and
$K=1|15$, $I\cdot K = 21|11|15$, $I\triangleright K=21|2|15$, and
$I|K=21|1|1|15$.

\medskip

A \emph{segmented permutation} is a permutation where the values
can be separated by bars. We denote by $\SP_n$ the set of segmented
permutations of size $n$. There are $2^{n-1}n!$ segmented permutations
of size $n$.

In this section, $\sigma$ denotes a segmented permutation of size
$n$. A position $i < n$ is called a \emph{segmentation} if there is a
bar between $\sigma_i$ and $\sigma_{i+1}$. A position $i$ is a
\emph{descent} if it is not a segmentation and
$\sigma_{i}>\sigma_{i+1}$. We denote by $\des(\sigma)$
(resp. $\seg(\sigma)$) the number of descents (resp. bars) of
$\sigma$. For example, with $\sigma = 3|7156|24$, we have
$(\des(\sigma), \seg(\sigma)) = (1, 2)$.
We define the \emph{segmented composition of descents} $\sigma$,
denoted by $\SCDes(\sigma)$,  as the segmented
composition of $n$ whose descent set corresponds to the descents of $\sigma$
and whose segmentation set corresponds to the segmentations of
$\sigma$. For example, with $\sigma =  3|7156|24$ and $I=(1|13|2)$, we have
$\SCDes(\sigma)=I$ as $\Des(I)=\{2\}$ is the only position of
descent in $\sigma$ and $\Bar(I)=\{1,5\}$ are the two positions
of segmentation of $\sigma$. Note that with $I=\SCDes(\sigma)$, these
definitions imply $\des(\sigma)=\des(I)-1$ and $\seg(\sigma)=\seg(I)$.

In~\cite{CN}, the authors defined \emph{partially signed permutations}
and a recoil statistic on those objects. There is a bijection
involving the inverse of signed permutations between segmented
permutations and partially signed permutations that sends this recoil
statistic to the segmented composition of descents of a segmented
permutation.

\subsection{Segmented composition quasi-symmetric functions: $\SCQSym$}

In~\cite{SCQSym}, the authors defined an algebra over segmented
compositions ($\SCQSym$) and a so called ribbon basis ($\RSC_I$). This
basis satisfies the following product rule:
\begin{equation}
  \RSC_I \cdot \RSC_{K} = \RSC_{I \cdot K} + \RSC_{I |K} +
  \RSC_{I\triangleright K}.
\end{equation}
For example,
$\RSC_{21|1}\cdot\RSC_{2|15}=\RSC_{21|12|15}+\RSC_{21|1|2|15}+\RSC_{21|3|15}$.

We define a multiplicative basis of $\SCQSym$ as an analogue of the
complete functions of the noncommutative symmetric functions~\cite{NCSF}
\begin{equation}
  \SSC^I = \sum_{I \succeq K}\RSC_K
\end{equation}
The inverse relation is given by
\begin{equation}
  \label{eq:RtoS}
  \RSC_I = \sum_{I\succeq K}(-1)^{\des(I) - \des(K)}\SSC^K.
\end{equation}
We have for example $\SSC^{2|13}=\RSC_{2|13}+\RSC_{2|4}+\RSC_{2|1|3}$
and $\RSC_{2|13} = \SSC^{2|13}-\SSC^{2|4}-\SSC^{2|1|3}$.

The multiplication rule for this basis is the following.
\begin{proposition}
  Let $I$ and $K$ be two segmented compositions.
  \begin{equation}
    \SSC^I \cdot \SSC^K = \SSC^{I\cdot K}
  \end{equation}
\end{proposition}

\subsection{Segmented permutation quasi-symmetric functions: $\SPQSym$}

Recall that the \emph{standardization} of a word $w$ is the
permutation obtained by iteratively scanning $w$ from left to right,
and labeling $1, 2, \cdots$ the occurrences of its smallest letter,
then numbering the occurrences of the next one, and so on. We define
the standardization for segmented word by standardizing the underlying
word and keeping the bars in the same positions. For example,
$\std(41|2116|4)= 51|4237|6$.

Given two segmented permutations, $\sigma\in\SP_n$ and $\tau\in\SP_r$,
define the \emph{convolution} of $\sigma$ and $\tau$ (denoted by
$\sigma *\tau$) as the set of all segmented permutations in
$\SP_{n+r}$ such that the standardization of the $n$ first letters is
equal to $\sigma$ and the standardization of the $r$ last letters
gives $\tau$. For example,
$2|13*12=\{2|1345,2|13|45,2|1435,2|14|35,\cdots,4|3512,4|35|12\}$. The
number of segmented permutations in $\sigma *\tau$ is
$2\binom{n+r}{n}$.

In~\cite{CN} we defined an algebra over \emph{partially signed
  permutations} which can be seen as an algebra over segmented
permutations that we call $\SPQSym$. This identification defines a
basis $\GSP_\sigma$ in $\SPQSym$ with the following product rule:
\begin{equation}
  \GSP_\sigma \cdot \GSP_\tau = \sum_{\mu \in \sigma * \tau}\GSP_\mu.
\end{equation}

The algebra $\SCQSym$ can be seen as a subalgebra of $\SPQSym$ with
the following identification.
\begin{equation}
  \RSC_I = \sum_{\SCDes(\sigma)=I} \GSP_\sigma.
\end{equation}
For example,
$R_{2|11}=\GSP_{12|43}+\GSP_{13|42}+\GSP_{14|32}+\GSP_{23|41}+\GSP_{24|31}+\GSP_{34|21}$.

\section{Generalized Eulerian numbers}
\label{sec:numbers}
\subsection{Generalized Eulerian triangle}

Let us now consider the triangle of numbers $T(n,k)$ corresponding to the
number of segmented permutations of size $n$ having $k$ descents. In
other words
\begin{equation}
  T(n, k) = \#\{\sigma \in \SP_n ~|~ \des(\sigma) = k\}.
\end{equation}

The first values of the triangle corresponding to the numbers $T(n,k)$
are presented in Figure~\ref{fig:triangle} alongside with the usual
Eulerian numbers.

\begin{figure}[h!t]
  \begin{equation*}
    \begin{array}{r|cccccc}
      n\backslash r & 0 & 1 & 2 & 3 & 4 & 5 \\ \hline
      0 & 1 \\
      1 & 1 \\
      2 & 3 & 1 \\
      3 & 13 & 10 & 1 \\
      4 & 75 & 91 & 25 & 1 & \\
      5 & 541 & 896 & 426 & 56 & 1 \\
      6 & 4683 & 9829 & 6734 & 1674 & 119 & 1
    \end{array} \qquad
    \begin{array}{r|cccccc}
      n\backslash r & 0 & 1 & 2 & 3 & 4 & 5 \\ \hline
      0 & 1 \\
      1 & 1 \\
      2 & 1 & 1 \\
      3 & 1 & 4 & 1 \\
      4 & 1 & 11 & 11 & 1 & \\
      5 & 1 & 26 & 66 & 26 & 1 \\
      6 & 1 & 57 & 302 & 302 & 57 & 1
    \end{array}
  \end{equation*}
  \caption{Triangles of generalized Eulerian numbers on the left
    and usual Eulerian numbers on the right.}
  \label{fig:triangle}
\end{figure}

The numbers appearing on the first column of this triangle are known
as the ordered Bell numbers, sequence~\OEIS{A000670}
of~\cite{Slo}. Among other things, they count the 
number of ordered set partitions of size $n$. They are in bijection
with segmented permutations with no descents by considering the sets
of values delimited by the bars in a segmented permutation. These
numbers are also equal to the Eulerian polynomials evaluated at $t=2$,
$A_n(2)$. 

We can also give a combinatorial interpretation of the rows of the
triangle read from right to left.

\begin{proposition}
  \label{prop:mirror}
  Let $m<n$ be two positive integers, we have
  \begin{equation}
    T(n, n-m-1) = \#\{\sigma \in \SP_n ~|~ \seg(\sigma) + \des(\sigma) = m\}.
  \end{equation}
\end{proposition}

\begin{proof}
  Let $\sigma \in \SP_n$ having $\seg(\sigma)+\des(\sigma)=m$
  descents and bars. In the mirror image of $\sigma$, there
  are exactly $n-1-m$ descents which correspond to the positions that
  are neither a descent nor a segmentation in $\sigma$.
\end{proof}

We can refine the triangle in Figure~\ref{fig:triangle} by considering
the 3-dimensional tetrahedron consisting in the numbers

\begin{equation}
  K(n,i,j) = \#\{\sigma \in \SP_n~|~\des(\sigma)=i,~\seg(\sigma)=j\},
\end{equation}
so that
\begin{equation}
  \label{eq:TK}
  T(n,k) = \sum_{j=0}^{n-1-k}K(n,k,j).
\end{equation}

In Figure~\ref{fig:3dtriangle}, we represent some values for the
numbers $K(n,i,j)$ where we fix the size of the segmented
permutations.
\begin{figure}[h!t]
  \begin{equation*}
    n=2:~\begin{array}{r|cc}
      j\backslash i & 0 & 1  \\ \hline
      0 & 1 & 1\\
      1 & 2 \\
    \end{array}
    \qquad n=3:~
    \begin{array}{r|ccc}
      j\backslash i & 0 & 1 & 2 \\ \hline
      0 & 1 & 4 & 1\\
      1 & 6 & 6 \\
      2 & 6
    \end{array}
    \qquad n=4:~
    \begin{array}{r|cccc}
      j\backslash i & 0 & 1 & 2 & 3  \\ \hline
      0 & 1  & 11 & 11 & 1\\
      1 & 14 & 44 & 14 \\
      2 & 36 & 36 \\
      3 & 24
    \end{array}
  \end{equation*}
  \caption{Slices of the tetrahedron of refined generalized Eulerian
    numbers.}
  \label{fig:3dtriangle}
\end{figure}

We represent these triangles with the parameter $i$ for the columns
and the parameter $j$ for the rows such that the first row of the
$n$-th triangle corresponds to the $n$-th row of the Eulerian
triangle. Moreover, by considering the mirror image of a segmented
permutation we have a straightforward proof of the following
property corresponding to the symmetry of the rows of the triangles.

\begin{proposition}
  \label{prop:mirrorK}
  Let $n > 0$ and $0\leq i+j<n$, we have
  \begin{equation}
    K(n,i,j) = K(n,n-j-i,j).
  \end{equation}
\end{proposition}

The numbers appearing in the first column of the triangles are known
as sequence \OEIS{A019538} of~\cite{Slo}.
\begin{equation}
  \label{eq:stirling}
  K(n,0,j) = (j+1)!S(n,j+1),
\end{equation}
where $S(n,k)$ are the Stirling numbers of the second kind, sequence
\OEIS{A008277} of~\cite{Slo}, which count the number of ways to
partition the set $\{1,2,\cdots,n\}$ into $k$ subsets. In fact for all $n$ and
$j<n$, the $j$-th row of the $n$-th triangle can be divided by
$(j+1)!$ as the order of the blocks of numbers between the bars of a
segmented permutation does not change the number of descents.

The numbers $K(n,i,j)$ can be described recursively, by decreasing
either $n$ or $j$.

\begin{proposition}
  For $n > 0$ and $0\leq i+j<n$,
  \begin{equation}
    \label{eq:recKn}
    \begin{split}
    K(n,i,j) =(&i+j+1)\Big{[}K(n-1,i,j) + K(n-1,i,j-1)\Big{]} \\
    & + (n-i-j)\Big{[}K(n-1,i-1,j) + K(n-1,i-1,j-1)\Big{]}.
    \end{split}
  \end{equation}
  For $j > 0$,
  \begin{equation}
    \label{eq:recKj}
    jK(n,i,j) = (n-i-j)K(n, i, j-1) + (i+1)K(n,i+1,j-1).
  \end{equation}
\end{proposition}

\begin{proof}
  The first equation of the proposition is proved bijectively based on
  the different possibilities we have to add an $n$ in a segmented
  permutation of $n-1$. The second one is obtained by considering what
  happens when we add a bar in a segmented permutation.
\end{proof}

Using~\eqref{eq:recKn}, one obtains a similar recurrence
for the numbers $T(n,k)$:
\begin{corollary} Let $n > 0$ and $0 \leq k < n$. 
  \begin{equation}
    T(n,k) = (n-k)T(n-1,k-1) + (n+1)T(n-1,k) + (k+1)T(n-1,k+1).
  \end{equation}
\end{corollary}
With an induction on~\eqref{eq:recKj}, we
obtain the following corollary expressing $K(n,i,j)$ in terms of
Eulerian numbers.
\begin{corollary}
  Let $n > 0$ and $0\leq i+j<n$.
  \begin{equation}
    K(n,i,j) = \sum_{k=0}^{n-1}\binom{k}{i}\binom{n-1-k}{i+j-k}A(n,k).
  \end{equation}
\end{corollary}
Note that this result can also be proved directly using the
combinatorial interpretation.

\subsection{Generalized Eulerian polynomials}

A different approach to study the numbers $K(n,i,j)$ is to study the
associated polynomials:
\begin{equation}
  \alpha_n(t,q) = \sum_{\sigma \in \SP_n} t^{\des(\sigma)}q^{\seg(\sigma)}.
\end{equation}

Let us denote the generating polynomials of the triangle in
Figure~\ref{fig:triangle} by $P_n(t)$.  We have 
\begin{equation}
  P_n(t) = \sum_{\sigma \in \SP_n} t^{\des(\sigma)}.
\end{equation}

Different values of the parameters $t$ and $q$ in $\alpha_n$ give
known polynomials or values:

\begin{proposition}
  Let $n \geq 0$.
  \begin{multicols}{2}
    \begin{enumerate}[parsep=0cm,itemsep=0.2cm,topsep=0cm]
    \item $\alpha_n(t,0) = A_n(t)$,
    \item $\alpha_n(0,q) = B_n(q)$,
    \item $\alpha_n(t,1) = P_n(t)$,
    \item $\alpha_n(t,t) = t^{n-1}P_n\left(\frac{1}{t}\right)$,
    \item $\alpha_n(-1,1) = 2^{n-1}$,
    \end{enumerate}
  \end{multicols}
  where $A_n(t)$ are the Eulerian polynomials and $B_n(t)$ are the
  ordered Bell polynomials.
\end{proposition}

\begin{proof}
Items 1 and 3 come from the combinatorial interpretation. By
definition,
\begin{equation}
  B_n(q) = \sum_{r=0}^{n-1}S(n,r+1)(r+1)!q^r
\end{equation}
so item 2 is proved using~\eqref{eq:stirling}. Item 4 is a consequence
of Proposition~\ref{prop:mirror}. Finally, the last item can be
proved bijectively but we shall only consider the proof given by
substituting $t$ to $-1$ and $q$ to $1$ in the generating function of
the $\alpha_n$ that we make explicit in Theorem~\ref{th:GF}.
\end{proof}

The symmetry described by Proposition~\ref{prop:mirrorK} implies the
following relation on the polynomials.
\begin{proposition}
  Let $n \geq 0$.
  \begin{equation}
    \alpha_n(t,q) = t^{n-1}\alpha_n\left(\frac{1}{t},\frac{q}{t}\right)
  \end{equation}
\end{proposition}

Using~\eqref{eq:recKj}, we obtain the following differential recurrence
relation satisfied by the~$\alpha_n$.
\begin{proposition}
  Let $n \geq 2$.
  \begin{equation}
    \label{eq:recDiff}
    \begin{split}
    \alpha_n(t,q) = \Big{(}& (n-2)tq + (n-1)t + 2q + 1\Big{)}\alpha_{n-1}(t,q)\\
    & + (t-t^2)(q+1)\frac{\partial}{\partial t}\alpha_{n-1}(t,q) \\
    & + (1-t)(q^2+q)\frac{\partial}{\partial q}\alpha_{n-1}(t,q)
    \end{split}
  \end{equation}
\end{proposition}

Let $G(t,q,x)$ be the exponential generating function of the polynomials
$\alpha_n$:
\begin{equation}
  G(t,q,x) = \sum_{n\geq 0}\alpha_n(t,q)\frac{x^n}{n!}.
\end{equation}

Using~\eqref{eq:recDiff} it is possible to obtain a differential
equation satisfied by $G$ but using it to obtain an explicit form of
$G$ seems at least as difficult as solving the equivalent problem on
the usual Eulerian polynomials. We shall obtain it easily and directly
by means of a noncommutative analog.

\section{Algebraic study}
\label{sec:alge}

The results in this section are inspired from a similar approach to 
Eulerian polynomials presented in~\cite{NCSF}.

\subsection{Definition and generating function}

Let $t$ and $q$ be indeterminates that commute with the
$\GSP_\sigma$ of $\SPQSym$. We define the \emph{generalized noncommutative Eulerian
  polynomials} as follows.
\begin{equation}
  \label{eq:defNC}
  \Acal_n(t,q)=\sum_{\sigma\in\SP_n}t^{\des(\sigma)+1}q^{\seg(\sigma)}\GSP_\sigma
  = \sum_{I \segn n}t^{\des(I)}q^{\seg(I)}\RSC_I.
\end{equation}

Let $n > 0$ and $\sigma\in\SP_n$, consider the following morphism of
algebras $\varphi$ defined by
\begin{equation}
  \varphi(\GSP_\sigma) = \frac{x^n}{2^{n-1}n!}.
\end{equation}
For $n > 0$, we have
\begin{equation}
  \varphi(\Acal_n(t,q)) = t\alpha_n(t,q)\frac{x^n}{2^{n-1}n!}.
\end{equation}

To obtain a more suitable expression we expand the $\Acal_n$ on the
complete functions.

\begin{proposition}
  Let $n \geq 0$.
  \begin{equation}
    \label{eq:AnS}
    \Acal_n(t,q) = \sum_{I\segn n} t^{\des(I)}(1-t)^{n-\ell(I)}(q-t)^{\seg(I)}S^I.
  \end{equation}
\end{proposition}

\begin{proof}
In \eqref{eq:defNC}, we expand the ribbon basis on the complete one
using \eqref{eq:RtoS} which gives us:
\begin{equation}
  \begin{split}
    \Acal_n(t,q) = & \sum_{I\segn n} \sum_{I \succeq K}t^{\des(K)}q^{\seg(K)}
    (-t)^{\des(I)-\des(K)}q^{\seg(I)-\seg(K)}S^K \\
    = & \sum_{K\segn n} t^{\des(K)}q^{\seg(K)}\left[\sum_{u=0}^{n-\ell(K)}\sum_{v=0}^{\seg(K)}
      \binom{n-\ell(K)}{u}(-t)^u\binom{\seg(K)}{v}
      \left(\frac{-t}{q}\right)^v\right]S^K.
  \end{split}
\end{equation}
To obtain the last equality, one needs to count how many segmented
compositions of $n$ are finer than $K$ with a length
$\ell(K)+u$ and $v$ bars less that $K$.
The statement is obtained after applying twice the binomial theorem.
\end{proof}

Equation~\eqref{eq:AnS} can be rewritten as:
\begin{equation}
  \Acal_n(t,q) = (1-t)^n\sum_{I\segn n} \left(\frac{t}{1-t}\right)^{\des(I)}
  \left(\frac{q-t}{1-t}\right)^{\seg(I)}S^I.
\end{equation}
Before applying $\phi$ to this relation, we need to define the
following series of complete functions.
\begin{equation}
  \Pi_{i,j} = \sum_{n \geq 1}\sum_{\substack{I\segn n\\ \des(I) = i\\ \seg(I) = j}}S^I.
\end{equation}

The image by $\phi$ of the $\Pi_{1,u}$ is given by the following lemma.
\begin{lemma}
  \label{lem:phiPi}
  Let $u\geq 0$.
  \begin{equation}
    \phi(\Pi_{1,u}) = 2\left(e^{x/2}-1\right)^{u+1}
  \end{equation}
\end{lemma}

\begin{proof}
  We expand $\Pi_{1,u}$ in $\SPQSym$,
  \begin{equation}
    \Pi_{1,u} = \sum_{n\geq 1}
    \sum_{\substack{\sigma \in\SP_n\\ \des(\sigma)=0 \\ \seg(\sigma)=u}}\GSP_\sigma.
  \end{equation}
  By applying $\phi$ and then~\eqref{eq:stirling} we obtain
  \begin{equation}
    \begin{split}
      \phi(\Pi_{1,u}) = &\sum_{n\geq 1}K(n,0,u)\frac{x^n}{2^{n-1}n!}\\
      = & ~2\sum_{n \geq 1}S(n, u+1)(u+1)!\frac{(x/2)^n}{n!}
    \end{split}
  \end{equation}
  Using the fact that
  $\displaystyle \sum_{n \geq 1}S(n,u)u!\frac{x^n}{n!}=(e^x-1)^u$, we
  obtain the statement.
\end{proof}

We can now obtain the generating function for the polynomials $\alpha_n$.

\begin{theorem}
  \label{th:GF}
  We have the following generating function:
  \begin{equation}
    G(t,q,x) = 1 + \frac{e^{x(1-t)} - 1}{1+q - (t+q)e^{x(1-t)}}.
  \end{equation}
\end{theorem}

The proof of the theorem is straightforward from the following lemma.
\begin{lemma}
  \label{lem:GF}
  We have
  \begin{equation}
    1+\sum_{n\geq 1}\frac{t\alpha_n(t,q)}{(1-t)^n}\frac{x^n}{2^{n-1}n!} 
    = \frac{(1-t) - (q-t)(e^{x/2}-1)}{(1-t) - (q+t)(e^{x/2}-1)}.
  \end{equation}
\end{lemma}

\begin{proof}
  Let us consider the series:
  \begin{equation}
    \begin{split}
      \sum_{n\geq 0}\frac{\Acal_n(t,q)}{(1-t)^n} = &
      \sum_{n \geq 0} \sum_{I\segn n}\left(\frac{t}{1-t}\right)^{\des(I)}
      \left(\frac{q-t}{1-t}\right)^{\seg(I)}S^I \\
      = & \sum_{v\geq 0}\left(\frac{t}{1-t}\right)^v\left[\sum_{u \geq 0}
        \left(\frac{q-t}{1-t}\right)^u\Pi_u\right]^v
    \end{split}
  \end{equation}
  Then, by applying $\phi$ and using Lemma~\ref{lem:phiPi} we obtain
  \begin{equation}
      1 + \sum_{n\geq 1}\frac{t\alpha_n(t,q)}{(1-t)^n}\frac{x^n}{2^{n-1}n!}
      = \sum_{v\geq 0}\left(\frac{2t}{1-t}\right)^v\left[\sum_{u \geq 0}
        \left(\frac{q-t}{1-t}\right)^u\left(e^{x/2}-1\right)^u\right]^v
  \end{equation}
  This equation can be rewritten to obtain the lemma as both
  series on the right are geometric.
\end{proof}

As a corollary of Theorem~\ref{th:GF}, one describes the generating
functions of the polynomials $P_n(t)$.

\subsection{Applications}

By substituting $q$ to $1$ and $x$ to $2x$ in Lemma~\ref{lem:GF} and
dividing each side by $(1-t)$, we obtain
\begin{equation}
  \frac{1}{1-t} +
  2\sum_{n\geq 1}\frac{t\alpha_n(t,1)}{(1-t)^{n+1}}\frac{x^n}{n!}
  = 1 + t\frac{e^x}{2-(1+t)e^x}.
\end{equation}
As $\alpha_n(t,1) = P_n(t)$, this equation leads to the following.
\begin{proposition}
  \label{prop:Dobin}
  Let $n > 0$, we have
  \begin{equation}
    \frac{P_n(t)}{(1-t)^{n+1}} = \sum_{k \geq 0}(1+t)^{k-1}\frac{k^n}{2^{k-1}}.
  \end{equation}
\end{proposition}
This equation gives a generalization of the similar result for the
Eulerian numbers. Moreover, by specializing $t$ to $0$, we recover a
known expression of the ordered Bell numbers as a series.

Another result can be obtained by considering the image under $\phi$ of the
component of degree $n$ of the $\Pi_{i,j}$ and applying $\phi$ to~\eqref{eq:AnS}:
\begin{proposition}
  Let $n > 0$,
  \begin{equation}
    \alpha_n(t,q) = \sum_{0\leq i+j\leq n-1}t^i(q-t)^j(1-t)^{n-i-j-1}
    2^i(i+j+1)!\binom{i+j}{j}S(n,i+j+1).
  \end{equation}
\end{proposition}

\subsection{Identities of Worpitzky}

The identities of Worpitzky are well-known identities involving Eulerian
numbers.
\begin{equation*}
  k^n = \sum_{i=0}^{k-1}\binom{k+n-i-1}{n}A(n,i).
\end{equation*}

Consider the discrete derivation of polynomials defined on monomials as
\begin{equation}
  \Delta(X^n) = (X+1)^n - X^n.
\end{equation}
Applying this derivation to Worpitzky's identities gives
\begin{equation}
  \Delta^r(k^n) = \sum_{i=0}^{k-1}\binom{k+n-i-1}{n-r}A(n,i).
\end{equation}

Worpitzky's identities are usually proved by explicitly describing the
coefficient of the series
$\displaystyle \sum_{n\geq 0} \frac{A_n(t)}{(1-t)^n}\frac{x^n}{n!}$. By doing so
with the polynomials $\alpha_n(t,q)$, we obtain the following
identity.

\begin{theorem}
  \label{th:Worp}
  Let $n$, $k$, and $r$ be three positive integers.
  \begin{equation}
    \binom{k+r-1}{r}\Delta^{r+1}((k-1)^n) =
    \sum_{i=0}^{k-1}\binom{n+k-i}{n-1}K(n,i,r).
  \end{equation}
\end{theorem}

\section*{Conclusion}

In this paper we defined generalized Eulerian numbers and generalized
Eulerian polynomials which appear to interact with several known
sequences of integers. We presented here some of them, the
specialization of the polynomials $\alpha_n(t,q)$ at small
values ($t,q = -2, -1, 1, 2, \ldots$) gives other sequences
of~\cite{Slo}.

We also noticed that the rows and columns of the triangles and
tetrahedron of the $T(n,k)$ and $K(n,i,j)$ are unimodal sequences. We
verified these conjectures using \emph{sage} up to $n=1000$.


\printbibliography
\end{document}